\documentclass[reqno]{amsart}
\usepackage{amssymb}
\usepackage[colorlinks=false,hypertex,pagebackref]{hyperref}%
\allowdisplaybreaks[4]
\theoremstyle{plain}
\newtheorem{thm}{Theorem}
\newtheorem{lem}{Lemma}

\theoremstyle{remark}
\newtheorem{rem}{Remark}
\DeclareMathOperator{\td}{d\mspace{-2mu}}
\date{Complete on 3 January 2010 in Tianjin Polytechnic University}
\date{}

\begin{document}

\title[Two monotonic functions involving gamma function]
{Two monotonic functions involving gamma function and volume of unit ball}

\author[F. Qi]{Feng Qi}
\address[F. Qi]{Department of Mathematics, College of Science, Tianjin Polytechnic University, Tianjin City, 300160, China}
\email{\href{mailto: F. Qi <qifeng618@gmail.com>}{qifeng618@gmail.com}, \href{mailto: F. Qi <qifeng618@hotmail.com>}{qifeng618@hotmail.com}, \href{mailto: F. Qi <qifeng618@qq.com>}{qifeng618@qq.com}}
\urladdr{\url{http://qifeng618.spaces.live.com}}

\author[B.-N. Guo]{Bai-Ni Guo}
\address[B.-N. Guo]{School of Mathematics and Informatics, Henan Polytechnic University, Jiaozuo City, Henan Province, 454010, China}
\email{\href{mailto: B.-N. Guo <bai.ni.guo@gmail.com>}{bai.ni.guo@gmail.com},
\href{mailto: B.-N. Guo <bai.ni.guo@hotmail.com>}{bai.ni.guo@hotmail.com}}

\begin{abstract}
In present paper, we prove the monotonicity of two functions involving the gamma function $\Gamma(x)$ and relating to the $n$-dimensional volume of the unit ball $\mathbb{B}^n$ in $\mathbb{R}^n$.
\end{abstract}

\keywords{monotonic function, gamma function, volume, unit ball, Descartes' Sign Rule}

\subjclass[2000]{Primary 33B15; Secondary 26A48}

\thanks{The authors were supported in part by the Science Foundation of Tianjin Polytechnic University}

\thanks{This paper was typeset using \AmS-\LaTeX}

\maketitle

\section{Introduction}

It is well-known that the classical Euler's gamma function may be defined by
\begin{equation}\label{egamma}
\Gamma(x)=\int^\infty_0t^{x-1} e^{-t}\td t
\end{equation}
for $x>0$ and that the $n$-dimensional volume of the unit ball $\mathbb{B}^n$ in $\mathbb{R}^n$ is denoted by
\begin{equation}
\Omega_n=\frac{\pi^{n/2}}{\Gamma(1+n/2)},\quad n\in\mathbb{N}.
\end{equation}
\par
For $x\ge0$, define the function
\begin{equation}\label{gamma-ln-ratio}
F(x)=\begin{cases}
\dfrac{\ln\Gamma(x+1)}{\ln(x^2+1)-\ln(x+1)},&x\ne0,1,\\
\gamma,&x=0,\\
2(1-\gamma),&x=1.
\end{cases}
\end{equation}
Recently, the function $F(x)$ was proved in \cite{refine-Ivady-gamma.tex} to be strictly increasing on $[0,1]$. Moreover, as a remark in \cite{refine-Ivady-gamma.tex}, the function $F(x)$ was also conjectured to be strictly increasing on $(1,\infty)$.
\par
The first aim of this paper is to verify above-mentioned conjecture which can be recited as the following theorem.

\begin{thm}\label{mon-gamma-function-thm}
The function $F(x)$ defined by \eqref{gamma-ln-ratio} is strictly increasing on $[0,\infty)$.
\end{thm}

The second aim of this paper is to derive the monotonicity of the sequence
\begin{equation}\label{omega-ln-x-2-x}
\Omega_n^{1/[\ln(n^2/4+1)-\ln(n/2+1)]}
\end{equation}
for $n\in\mathbb{N}$ by establishing the following general conclusion.

\begin{thm}\label{mon-unit-ball-thm}
The function
\begin{equation}\label{mon-unit-ball-G(x)}
  G(x)=\biggl[\frac{\pi^{x}}{\Gamma(x+1)}\biggr]^{1/[\ln(x^2+1)-\ln(x+1)]}
\end{equation}
is strictly decreasing on $(1,\infty)$. Consequently, the sequence \eqref{omega-ln-x-2-x} is strictly decreasing for $n\ge3$.
\end{thm}

\section{Two lemmas}

In order to prove Theorem~\ref{mon-gamma-function-thm}, we need the following lemma which can be found in \cite[pp.~9--10, Lemma~2.9]{refine-jordan-kober.tex}, \cite[p.~71, Lemma~1]{elliptic-mean-comparison-rev2.tex} or closely-related references therein.

\begin{lem}\label{hospital-rule-ratio}
Let $f$ and $g$ be continuous on $[a,b]$ and differentiable on $(a,b)$ such
that $g'(x)\ne0$ on $(a,b)$. If $\frac{f'(x)}{g'(x)}$ is increasing $($or
decreasing$)$ on $(a,b)$, then so are the functions $\frac{f(x)-f(b)}{g(x)-g(b)}$ and
$\frac{f(x)-f(a)}{g(x)-g(a)}$ on $(a,b)$.
\end{lem}

We also need the following elementary conclusions.

\begin{lem}\label{0<>0}
The functions
\begin{align*}
p_1(x)&=x^3+3 x^2-x-1,\\
p_2(x)&=x^3+3x^2-x-1,\\
p_3(x)&=3 x^4+8 x^3+2 x^2-1,\\
p_4(x)&=x^5+3 x^4+2 x^3+2 x^2+x-1,\\
p_5(x)&=x^5+5 x^4+6 x^3-3 x-1,\\
p_6(x)&=120 (15-4 \ln\pi)x^3 +240 (20-7 \ln\pi) x^2\\
  &\quad+48  (59-32 \ln\pi)x+72 (3-4 \ln\pi)
\end{align*}
are positive on $(1,\infty)$.
\end{lem}

\begin{proof}
An easy calculation shows that
\begin{align*}
15-4 \ln\pi&=10.421\dotsm,&20-7 \ln\pi&=11.986\dotsm,\\
59-32 \ln\pi&=22.368\dotsm,&3-4 \ln\pi&=-1.578\dotsm.
\end{align*}
Then Descartes' Sign Rule tells us that the function $p_i(x)$ for $1\le i\le6$ have just one possible positive root. Since
\begin{align*}
p_1(0)&=-1,& p_2(0)&=-1,& p_3(0)&=-1,& p_4(0)&=-1,&p_5(0)&=-1,\\
p_1(1)&=2,& p_2(1)&=2,& p_3(1)&=12,& p_4(1)&=8,&p_5(1)&=8,
\end{align*}
and
\begin{align*}
p_6(0)&=-72 (4 \ln\pi-3)\\
&=-113.68\dotsm,\\
p_6(1)&=-[120 (4 \ln\pi-15)+72 (4 \ln\pi-3)\\
&\quad+240 (7 \ln\pi-20)+48 (32 \ln\pi-59)]\\
&=5087.39\dotsm,
\end{align*}
these functions are positive on $[1,\infty)$.
\end{proof}

\section{Proof of Theorem~\ref{mon-gamma-function-thm}}

The monotonicity of the function $F(x)$ on $[0,1]$ was proved in \cite{refine-Ivady-gamma.tex}.
\par
For $x\in[1,\infty)$, it is easy to see that
\begin{equation}\label{ratio-gamma-ln}
\frac{\ln\Gamma(x+1)}{\ln(x^2+1)-\ln(x+1)} =\frac{\ln\Gamma(x+1)-\ln\Gamma(1+1)}{\ln\frac{x^2+1}{x+1}-\ln\frac{1^2+1}{1+1}} =\frac{f(x)-f(1)}{g(x)-g(1)},
\end{equation}
where
\begin{gather*}
f(x)=\ln{\Gamma(x+1)}\,\quad \text{and}\quad g(x)=\ln{\frac{x^2+1}{x+1}}
\end{gather*}
on $[1,\infty)$. Easy computation and simplification yield
\begin{equation*}
\frac{f'(x)}{g'(x)}=\frac{(x+1) \bigl(x^2+1\bigr) \psi'(x+1)}{x^2+2 x-1}
\end{equation*}
and
\begin{equation*}
\frac{\td}{\td x}\biggl[\frac{f'(x)}{g'(x)}\biggr]=\frac{q(x)}{(x^2+2 x-1)^2},
\end{equation*}
where
\begin{align*}
  q(x)&=\bigl(x^4+4 x^3-2 x^2-4 x-3\bigr) \psi(x+1)\\
  &\quad+(x+1) \bigl(x^2+1\bigr) \bigl(x^2+2 x-1\bigr) \psi'(x+1)
\end{align*}
and
\begin{align*}
q'(x)&=4 \bigl(x^3+3 x^2-x-1\bigr) \psi(x+1)\\
&\quad+\bigl(x^2+2
   x-1\bigr) \bigl[2\bigl(3 x^2+2 x+1\bigr) \psi'(x+1)\\
&\quad+(x+1)\bigl(x^2+1\bigr)\psi''(x+1)\bigr].
\end{align*}
By virtue of
\begin{equation}\label{qi-psi-ineq-1}
\ln x-\frac1x<\psi(x)<\ln x-\frac1{2x}
\end{equation}
and
\begin{equation}\label{qi-psi-ineq}
\frac{(k-1)!}{x^k}+\frac{k!}{2x^{k+1}} <\bigl\vert\psi^{(k)}(x)\bigr\vert <\frac{(k-1)!}{x^k}+\frac{k!}{x^{k+1}}
\end{equation}
for $x>0$ and $n\in\mathbb{N}$, see \cite[p.~131]{subadditive-qi.tex}, \cite[Lemma~3]{theta-new-proof.tex-BKMS}, \cite[p.~79]{AAM-Qi-09-PolyGamma.tex}, \cite[Lemma~3]{subadditive-qi-guo-jcam.tex} or related texts in \cite{Extension-TJM-2003.tex, Sharp-Ineq-Polygamma.tex}, and by using of the positivity of $p_1(x)$ in Lemma~\ref{0<>0} and
\begin{equation}\label{log-ineq-qi}
\frac{2t}{2+t}\le\ln(1+t)\le\frac{t(2+t)}{2(1+t)}
\end{equation}
on $(0,\infty)$, see \cite{Qi-Chen-MIQ-99} or \cite[p.~245, Remark~1]{Mon-Two-Seq-AMEN.tex}, we obtain
\begin{align*}
  q'(x)&>4 \bigl(x^3+3 x^2-x-1\bigr) \biggl[\ln(x+1)-\frac1{x+1}\biggr]\\
&\quad+\bigl(x^2+2
   x-1\bigr) \biggl\{2\bigl(3 x^2+2 x+1\bigr) \biggl[\frac1{x+1}+\frac1{2(x+1)^2}\biggr]\\
&\quad-(x+1)\bigl(x^2+1\bigr)\biggl[\frac1{(x+1)^2}+\frac2{(x+1)^3}\biggr]\biggr\}\\
&=4\bigl(x^3+3 x^2-x-1\bigr)\biggl[\ln(x+1)+\frac{4+x-4 x^2+6 x^3+16 x^4+5 x^5} {4(x+1)^2(x^3+3x^2-x-1)}\biggr]\\
&\ge4\bigl(x^3+3 x^2-x-1\bigr)\biggl[\frac{2x}{x+2}+\frac{4+x-4 x^2+6 x^3+16 x^4+5 x^5} {4(x+1)^2(x^3+3x^2-x-1)}\biggr]\\
&=\frac{13 x^6+66 x^5+86 x^4+8 x^3-31 x^2-2 x+8}{(x+1)^2 (x+2)}
\end{align*}
on $[1,\infty)$. Because
\begin{multline*}
13 x^6+66 x^5+86 x^4+8 x^3-31 x^2-2 x+8= 13 (x-1)^6+144 (x-1)^5\\
+611 (x-1)^4+1272 (x-1)^3+1364 (x-1)^2+712(x-1)+148>0
\end{multline*}
on $[1,\infty)$, it follows that $q'(x)>0$, and so the function $q(x)$ is increasing, on $[1,\infty)$. Due to
$$
q(1)=8 \left(\frac{\pi ^2}{6}-1\right)-4 (1-\gamma )=3.468\dotsm,
$$
the function $q(x)$ is positive on $[1,\infty)$. Therefore,
$$
\frac{\td}{\td x}\biggl[\frac{f'(x)}{g'(x)}\biggr]>0
$$
on $[1,\infty)$, which means that the function $\frac{f'(x)}{g'(x)}$ is strictly increasing on $[1,\infty)$. Furthermore, from Lemma~\ref{hospital-rule-ratio} and the equation \eqref{ratio-gamma-ln}, it follows that the function~\eqref{gamma-ln-ratio} is strictly increasing on $[1,\infty)$. The proof of Theorem~\ref{mon-gamma-function-thm} is complete.

\section{Proof of Theorem~\ref{mon-unit-ball-thm}}
Taking the logarithm of the function $G(x)$ and differentiating yield
\begin{equation*}
  \ln G(x)=\frac{(\ln\pi)x-\ln\Gamma(x+1)}{\ln(x^2+1)-\ln(x+1)}
\end{equation*}
and
\begin{equation*}
[\ln G(x)]'=\frac{g(x)}{[\ln(x+1)-\ln(x^2+1)]^2},
\end{equation*}
where
\begin{align*}
g(x)&=[\ln\pi-\psi(x+1)]\ln\frac{x^2+1}{x+1} -\frac{x^2+2 x-1}{(x+1)(x^2+1)}[(\ln\pi)x-\ln\Gamma(x+1)]\\
&=\frac{x^2+2 x-1}{(x+1)(x^2+1)}\biggl\{\frac{(x+1)(x^2+1)[\ln\pi-\psi(x+1)]}{x^2+2 x-1}\ln\frac{x^2+1}{x+1}\\
&\quad-(\ln\pi)x+\ln\Gamma(x+1)\biggr\}\\
&\triangleq\frac{x^2+2 x-1}{(x+1)(x^2+1)}h(x),
\end{align*}
with
\begin{align*}
h'(x)&=\frac{\ln (x+1)-\ln \bigl(x^2+1\bigr)}{(x^2+2 x-1)^2} \bigl\{\bigl(x^4+4 x^3-2 x^2-4 x-3\bigr)[\psi(x+1)-\ln\pi]\\
&\quad+(x+1)\bigl(x^2+1\bigr) \bigl(x^2+2 x-1\bigr) \psi'(x+1)\bigr\}\\
&\triangleq h_1(x)
\end{align*}
and
\begin{align*}
  h_1'(x)&=4\bigl(x^3+3x^2-x-1\bigr) \psi(x+1)\\
  &\quad+2 \bigl(3 x^4+8 x^3+2 x^2-1\bigr)\psi'(x+1)\\
   &\quad+\bigl(x^5+3 x^4+2 x^3+2 x^2+x-1\bigr) \psi''(x+1)\\
   &\quad-4\bigl(x^3+3x^2-x-1\bigr) \ln\pi.
\end{align*}
Utilizing Lemma~\ref{0<>0} and employing \eqref{qi-psi-ineq-1}, \eqref{qi-psi-ineq} for $k=1$ and \eqref{log-ineq-qi} give
\begin{align*}
  h_1'(x)&>4\bigl(x^3+3x^2-x-1\bigr) \biggl[\ln(x+1)-\frac1{x+1}\biggr]\\
  &\quad+2 \bigl(3 x^4+8 x^3+2 x^2-1\bigr)\biggl[\frac1{x+1}+\frac1{2(x+1)^2}\biggr]\\
   &\quad-\bigl(x^5+3 x^4+2 x^3+2 x^2+x-1\bigr)\biggl[\frac1{(x+1)^2}+\frac2{(x+1)^3}\biggr]\\
   &\quad-4\bigl(x^3+3x^2-x-1\bigr) \ln\pi\\
&=\frac{1}{(x+1)^2}\bigl[(7-4\ln\pi)x^5+(26-20\ln\pi)x^4 -6(4\ln\pi-3) x^3\\
&\quad+(11+12\ln\pi)x-2+4 \ln (\pi )\\
&\quad+4 \bigl(x^5+5 x^4+6 x^3-3 x-1\bigr) \ln (x+1)\bigr]\\
&>\frac{1}{(x+1)^2}\biggl[(7-4\ln\pi)x^5+(26-20\ln\pi)x^4 -6(4\ln\pi-3) x^3\\
&\quad+(11+12\ln\pi)x-2+4 \ln (\pi )\\
&\quad+4 \bigl(x^5+5 x^4+6 x^3-3 x-1\bigr) \frac{2x}{x+2}\biggr]\\
&=-\frac{1}{(x+1)^2 (x+2)}\bigl[(4 \ln\pi-15)x^6 +4 (7 \ln\pi-20)x^5 \\
&\quad+2 (32 \ln\pi-59)x^4 +12(4 \ln\pi-3) x^3 \\
&\quad+(13-12 \ln\pi)x^2 -4 (3+7\ln\pi)x +4-8 \ln\pi\bigr]\\
&\triangleq -\frac{1}{(x+1)^2 (x+2)}h_2(x)
\end{align*}
and
\begin{align*}
  h_2'(x)&=6  (4 \ln\pi-15)x^5+20 (7 \ln\pi-20) x^4+8  (32 \ln\pi-59)x^3\\
  &\quad+36 (4 \ln\pi-3)x^2+2  (13-12 \ln\pi)x-4 (3+7 \ln\pi),\\
  h_2''(x)&=30 (4 \ln\pi-15)x^4 +80  (7 \ln\pi-20)x^3\\
  &\quad+24 (32 \ln\pi-59) x^2+72 (4 \ln\pi-3)x+2 (13-12 \ln\pi),\\
  h_2'''(x)&=120 (4 \ln\pi-15)x^3 +240 (7 \ln\pi-20) x^2\\
  &\quad+48  (32 \ln\pi-59)x+72 (4 \ln\pi-3).
\end{align*}
By Lemma~\ref{0<>0}, it follows that $h_3''(x)$ is negative on $[1,\infty)$, so $h_2''(x)$ is decreasing and $h_2'(x)$ is concave on $[1,\infty)$. Since
\begin{align*}
h_2''(1)&=2 (13-12 \ln\pi)+30 (4 \ln\pi-15)+72 (4 \ln\pi-3)\\
&\quad+80 (7 \ln\pi-20)+24 (32 \ln\pi-59)\\
   &=-1696.22\dotsm,
\end{align*}
the derivative $h_2''(x)$ is negative, and thus $h_2(x)$ is concave and $h_2'(x)$ is decreasing, on $[1,\infty)$. From
\begin{align*}
  h_2'(1)&=2 (13-12 \ln\pi)+6 (4 \ln\pi-15)+36 (4 \ln\pi-3)\\
  &\quad+20 (7 \ln\pi-20)-4 (3+7 \ln\pi)+8 (32 \ln\pi-59)\\
  &=-469.89\dotsm,
\end{align*}
it is immediately deduced that $h_2'(x)$ is negative and the function $h_2(x)$ is decreasing on $[1,\infty)$. Due to
\begin{align*}
  h_2(1)&=2-16 \ln\pi+12 (4 \ln\pi-3)+4 (7 \ln\pi-20)\\
  &\quad-4 (3+7 \ln\pi)+2(32 \ln\pi-59)\\
  &=-134.10\dotsm,
\end{align*}
we derive that the function $h_2(x)$ is negative on $(1,\infty)$, so $h_1'(x)>0$ and $h_1(x)$ is increasing on $(1,\infty)$. From
\begin{equation*}
  h_1(1)=8 \biggl(\frac{\pi ^2}{6}-1\biggr)-4 (1-\gamma )+4 \ln\pi=8.04\dotsm,
\end{equation*}
it follows that the function $h_1(x)$ is positive on $(1,\infty)$, and thus the derivative $h'(x)$ is negative and $h(x)$ is decreasing on $(1,\infty)$. Since $h(1)=-\ln\pi=-1.14\dotsm$, it follows that the function $h(x)$ is negative, that the function $g(x)$ is negative, and that the derivative $[\ln G(x)]'$ is negative on $(1,\infty)$. As a result, the function $G(x)$ is strictly decreasing on $(1,\infty)$.
\par
It is clear that the sequence \eqref{omega-ln-x-2-x} equals $G\bigl(\frac{n}2\bigr)$, so the sequence \eqref{omega-ln-x-2-x} decreases for $n>2$. The proof of Theorem~\ref{mon-unit-ball-thm} is complete.

\section{Remarks}

\begin{rem}
In~\cite[Lemma~2.40]{AVV-Expos-89}, it was proved that the sequence $\Omega_n^{1/n}$ decreases strictly to $0$ as $n\to\infty$, that the series $\sum_{n=2}^\infty\Omega_n^{1/\ln n}$ is convergent, and that
\begin{equation}\label{n-to-infty-Omega-n}
\lim_{n\to\infty}\Omega_n^{1/(n\ln n)}=e^{-1/2}.
\end{equation}
\par
In~\cite[Corollary~3.1]{Anderson-Qiu-Proc-1997}, it was obtained that that the sequence $\Omega_n^{1/(n\ln n)}$ is strictly decreasing for $n\ge2$.
\par
In \cite{unit-ball.tex}, it was procured that the sequence $\Omega_{n}^{1/(n\ln n)}$ is strictly logarithmically convex for $n\ge2$.
\par
By L'Hospital rule, we have
\begin{equation*}
  \lim_{x\to\infty}\ln G(x)=\lim_{x\to\infty}\frac{(\ln\pi)-\psi(x+1)}{(x^2+2 x-1)/(x^3+x^2+x+1)}=-\infty,
\end{equation*}
hence,
$$
\lim_{x\to\infty}G(x)=\lim_{x\to\infty}\biggl[\frac{\pi^x}{\Gamma (x+1)}\biggr]^{{1}/{\ln\bigl(\frac{x^2+1}{x+1}\bigr)}}=0
$$
and the sequence \eqref{omega-ln-x-2-x} converges to $0$ as $n\to\infty$.
\end{rem}

\begin{rem}
We conjecture that the sequence \eqref{omega-ln-x-2-x} and the function \eqref{mon-unit-ball-G(x)} are both logarithmically convex on $(1,\infty)$.
\end{rem}

\end{document}